\documentclass[journal,onecolumn]{IEEEtran} 

\IEEEoverridecommandlockouts                              
\usepackage{graphicx}           
\usepackage{dcolumn}            
\usepackage{bm}                 

\usepackage{graphics}           
\usepackage{epsfig}             
\usepackage{times}              
\usepackage{amsmath}
\usepackage{amssymb}
\usepackage{amsfonts}
\usepackage{latexsym}
\usepackage{amsmath}
\usepackage{mathrsfs}
\usepackage{wasysym}

\usepackage{amsfonts}
\usepackage{fancyhdr}
\usepackage{color}
\usepackage{subfigure}
\usepackage{enumerate}

\newtheorem{mydef}{Definition }
\newtheorem{mytheorem}{Theorem}
\newtheorem{mycorollary}[mytheorem]{Corollary}
\newtheorem{mylemma}[mytheorem]{Lemma}
\newtheorem{myclaim}[mytheorem]{Claim}
\newtheorem{myproposition}[mytheorem]{Proposition}

\newtheorem{remark}{Remark}

\begin{document}
%
\title{\LARGE{Flatness-based Deformation Control of an Euler-Bernoulli Beam with In-domain Actuation}
\footnote
{This paper is a preprint of a paper submitted to IET Control Theory and Applications and is subject to Institution of Engineering and Technology Copyright. If accepted, the copy of record will be available at IET Digital Library}
}
%
%
\author{Amir Badkoubeh,  Jun Zheng, and Guchuan Zhu
\thanks{A. Badkoubeh with the the Department of Electrical Engineering,
        Ecole Polytechnique de Montreal, P.O. Box 6079, Station Centre-Ville, Montreal, QC, Canada H3T 1J4.
        E-mail: {\tt\small amir.badkoubeh@polymtl.ca}}%
\thanks{J. Zheng is with the Department of Basic Course, Southwest Jiaotong University,
          Emeishan, Sichuan 614202, China.
        E-mail: {\tt\small zheng@xx}}%
\thanks{G. Zhu is with the Department of Electrical Engineering, Ecole Polytechnique
        Montreal, P.O. Box 6079, Station Centre-Ville,
        Montreal, QC, Canada H3T 1J4.
        E-mail: {\tt\small guchuan.zhu@polymtl.ca}}%
}

%



\maketitle

\begin{abstract}                           
This paper addresses the problem of deformation control of an Euler-Bernoulli beam with in-domain actuation. The proposed control scheme consists in first relating the system model described by an inhomogeneous partial differential equation to a target system under a standard boundary control form. Then, a combination of closed-loop feedback control and flatness-based motion planning is used for stabilizing the closed-loop system around reference trajectories. The validity of the proposed method is assessed through well-posedness and stability analysis of the considered systems. The performance of the developed control scheme is demonstrated through numerical simulations of a representative micro-beam.
\end{abstract}
\maketitle
\section{Introduction}\label{Sec: Introduction}
The present work addresses the control of an Euler-Bernoulli beam with in-domain actuation described by an inhomogeneous partial differential equation (PDE). A motivating example of such a problem arises from the deformation control of micro-beams, shown in Fig.~\ref{freediagram}. This device is a simplified case of deformable micro-mirrors that are extensively used in adaptive optics \cite{Ref7,Vogel:2006}. Due to technological restrictions in the design, the fabrication and the operation of micro-devices, design methods that may lead to control structures with a large number of sensors and actuators are not applicable to microsystems with currently available technologies.

One of the standard methods to deal with the control of inhomogeneous PDEs is to discretize the PDE model in space to obtain a system of lumped ordinary differential equations (ODEs) \cite{Bensoussan:2006,Morris:2010}. Then, a variety of techniques developed for the control of finite-dimensional systems can be applied. However, in addition to the possible instability due to the phenomenon of spillover \cite{Balas:1978}, the increase of modeling accuracy may lead to high-dimensional and complex feedback control structures, requiring a considerable number of actuators and sensors for the implementation. Therefore, it is of great interest to directly deal with the control of PDE models, which may result in control schemes with simple structures.

\begin{figure}[t]
  \centering
  \includegraphics[scale=0.12]{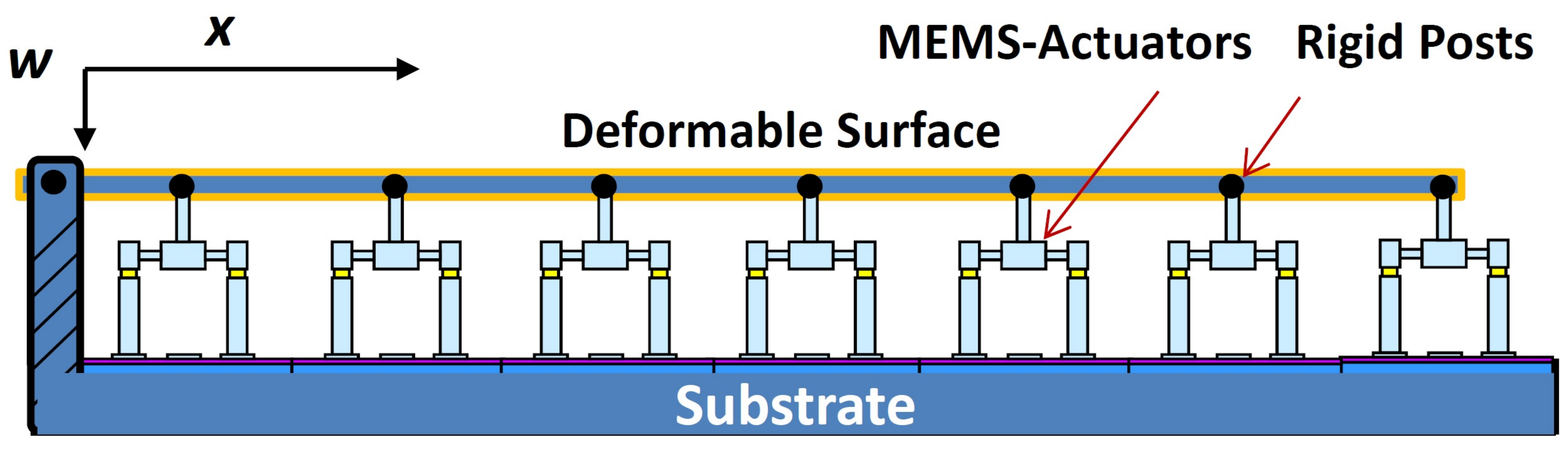}
  \caption{Schematic of the deformable microbeam.}
  \label{freediagram}
\end{figure}

There exists an extensive literature on the control of flexible beams described by PDEs in which the majority of the reported work deal with the stabilization problem by means of boundary control (see, e.g., \cite{Bensoussan:2006,Guo:2008,Krstic:2008,Luo:1999,He:2015_1,He:2015_2}). Placing actuators in the domain of the system will lead to inhomogeneous PDEs \cite{Bensoussan:2006,ammari:2000,Lasiecka:2000}. The stabilization problem of Euler-Bernoulli beams with in-domain actuation is considered in, e.g., \cite{ammari:2000,Ammari:2011,Banks:1996,Delfour:1987,Gall2007}. The deformation control of beams has gained in recent years an increasing attention and impotence (see, e.g., \cite{Banks:2002_TCST,Meurer:2013,Schorkhuber:2012_TCST}).

In the present work, we employ the method of flatness-based control, which has been applied to a variety of infinite-dimensional systems (see, e.g., \cite{Meurer:2013,Aoustin1997, Dunbar:2003,Fliess:1999-CDC,Laroch2000,Petit:2001,Petit2002,Rabbani:2011,Rudolph2003}). Nevertheless, applying this tool to systems controlled by multiple in-domain actuators leads essentially to a multiple-input multiple-output (MIMO) problem, which still remains a challenging topic.
It is noticed that a scheme proposed recently in \cite{Meurer:2013,Schorkhuber:2012_TCST} tackles this issue by utilizing the Weierstrass-factorized representation of the spectrum of the input-output dynamics. Nevertheless, an early truncation is still required in order to obtain a finite-dimensional input-output map, which is essential to apply the technique of flat systems.

The control scheme developed in this paper is based on an approach aimed at avoiding early truncations in control design procedure.
Specifically, it consists in first relating the original inhomogeneous model with pointwise actuation to a system in a standard boundary control form to which the technique of flatness-based motion planning may be applied for feedforward control. By using the technique of lifting, this boundary-controlled system can be transformed into a regularized inhomogeneous system driven by sufficiently smooth functions. It is shown that the steady-state solution of the regularized system can approximate that of the original system. A standard closed-loop feedback control is used to stabilize the original inhomogeneous system. Moreover, the MIMO issue is solved by splitting the reference trajectory into a set of sub-trajectories based on an essential property of the Green's function. Well-posedness and stability analysis of the considered systems is also carried out, which is essential for the validation of the developed control scheme.
It is worth noting that the approach proposed in this paper does not involve any early truncation. Moreover, it is shown that the invertibility of the resulting input-output map can be guaranteed \emph{a priori}. This allows for the control design to be carried out directly with the original PDE model in a systematic manner.

The remainder of the paper is organized as follows. Section~\ref{Sec: mathematical modelling} describes the model of a representative in-domain actuated deformable beam. Section~\ref{Sec: Mapping} introduces the relationship between the interior actuation and the boundary control and examines the well-posedness of the considered systems. Section~\ref{Sec: Stability} addresses the feedback control and stability issues. Section~\ref{Sec: Controller design} deals with motion planning for feedforward control. A simulation study is carried out in Section~\ref{simulation}, and, finally, some concluding remarks are presented in Section~\ref{Sec: Conclusion}.

\section{Model of the In-domain Actuated Euler-Bernoulli Beam}\label{Sec: mathematical modelling}
As shown in Fig.~\ref{freediagram}, the considered structure
consists of a continuous flexible beam and an array of actuators
connected to the beam via rigid spots. As the dimension of the spots
connecting the actuators to the continuous surface is much smaller
than the extent of the beam, the effect of the force generated by
actuators can be considered as a pointwise control represented by
Dirac delta functions concentrating at rigid spots. Furthermore, it is supposed that the structure is supported at the left end, terminated to the frame, and is suspended at the right end, vertically suspended on a moving actuator,
corresponding to a \emph{simply supported-shear hinged}
configuration \cite{Delfour:1987}.

As the developed scheme uses one actuator for stabilizing feedback control,
for notational convenience we consider a setup with $N+1$
actuators. The displacement of the beam at the position $x$ and the instance $t$
is denoted by $w(x,t)$. The derivatives of $w$ with respect to its variables are
denoted by $w_{x}$ and $w_{t}$, respectively. Suppose that the mass
density and the flexural rigidity of the beam are constant.
Then, the dynamic transversal displacement
of the beam with pointwise actuators located at $\{x_1, x_2 ,
\cdots, x_N, x_{N+1}\}$ in a normalized coordinate can be described by the following PDE
\cite{Banks:1996,Delfour:1987,Ref10}:
\begin{subequations}\label{indomain_form}
\begin{align}
  &w_{tt}(x,t)+ w_{\mathit{xxxx}}(x,t) = \sum_{j=1}^{N+1} \alpha_j(t)\delta(x-x_j), \;x\in \left(0,1\right), t > 0,\label{indomain_form1} \\
  &w(0,t)=w_{x}(1,t)=w_{\mathit{xx}}(0,t)=w_{\mathit{xxx}}(1,t) = 0, \;t> 0,\label{BC1}\\
  &w(x,0) = h_{0}(x), w_{t}(x,0)=h_{1}(x), \;x\in (0,1), \label{IV}
\end{align}
\end{subequations}
where $x$ is a normalized variable spanned over the domain $(0,1)$, $\delta(x-x_j)$ is the Dirac mass concentrated at the point
$x_j \in \left(0,1\right]$ and $\alpha_j : t
\mapsto \mathbb{R}$, $j = 1, \ldots, N+1$, are the control signals.
$h_0$ and $h_1$ represent the initial value of the beam. Without loss of generality, we
assume that $0<x_1<x_2<\cdots<x_N<x_{N+1}\leq 1$.

\begin{remark}
The considered problem can also be modeled as a serially connected
beam \cite{Delfour:1987}. Indeed, these two models are equivalent in
the sense that they lead to the same abstract linear system in
variational form (see, e.g., \cite{Rebarber:1995}).
\end{remark}

The control objective is to steer the beam to a desired form in steady state via in-domain pointwise actuation.

The model given in (\ref{indomain_form}) is a nonstandard PDE due to the unbounded inputs, represented by Dirac functions, on the right-hand side of
(\ref{indomain_form1}). Therefore, we need to invoke the weak derivative in the theory of distributions \cite{Pederson2000} and discuss solutions of the considered problem in a weak sense. To this end, we start by defining the convergence on $C^{\infty}_{0}(\Omega)$.

\begin{mydef}\label{def: convergence}
Let $\Omega$ be a domain in  $\mathbb{R}$. A sequence $\{\varphi_{j}\}$ of functions belonging to $C^{\infty}_{0}(\Omega)$ converge to $\varphi\in C^{\infty}_{0}(\Omega)$, if
\begin{enumerate}[(i)]
  \item there exists $K\subset \Omega$ such that $ \text{supp} (\varphi_{j}-\varphi )\subset K$ for every $j$, and
  \item $\lim_{j\rightarrow \infty} \partial ^n \varphi_{j}(x)=\partial^n\varphi(x)$ uniformly on $K$ for all $n\geq 0$.
\end{enumerate}
\end{mydef}

The linear space $C^{\infty}_{0}(\Omega)$ having the above property of convergence is called fundamental space, denoted by $\mathscr{D}(\Omega)$. The space of all linear continuous functionals on $ \mathscr{D}(\Omega)$, denoted by $\mathscr{D}'(\Omega)$, is called the space of (Schwartz) distributions on $\Omega$, which is the dual of $ \mathscr{D}(\Omega)$ (see, e.g., Chapter 1 of \cite{AF2003} for more properties of $\mathscr{D}(\Omega)$ and $\mathscr{D}'(\Omega)$).

Denote the set $\Phi$ by
\begin{equation}\label{space}
    \Phi = \left\{ \phi \in H^2\left(0,1\right) ; \phi(0)= \phi_x(1)=0 \right\}.
\end{equation}
We give the definition of weak solution of \eqref{indomain_form}.
\begin{mydef}\label{Def:weak_solution}
Let $T>0$ and $\alpha_{i}\in L^{2}\left(0,T\right)$ for $i=1,2,...,N+1$. Let $h_{0}\in \Phi,\ h_{1} \in L^{2}\left(0,1\right)$. A weak solution to the
problem~\eqref{indomain_form} is a function $w\in
C(\left[0,T\right]; \Phi)\cap C^1(\left[0,T\right]; L^2(0,1))$,
satisfying
\begin{align}
 w(x,0)= h_{0}(x),\ \ \ \ x\in \left(0,1\right),\notag
\end{align}
such that, for every $v\in C^1(\left[0,T\right]; \mathscr{D}(0,1))$, one has for almost every $\tau\in [0,T]$
\begin{equation}\label{Eq: weak_solution}
\begin{split}
&\int_{0}^{1}w_{t}(x,\tau)v(x,\tau)\text{d}x-\int_{0}^{1}h_{1}(x)v(x,0)
\text{d}x\\
& -\int_{0}^{\tau}\int_{0}^{1}w_{t}(x,t)v_{t}(x,t)\text{d}x\text{d}t\\
&+\int_{0}^{\tau}\int_{0}^{1}
w_{\mathit{xx}}(x,t)v_{\mathit{xx}}(x,t)\text{d}x\text{d}t \\
=& \sum_{i=1}^{N+1}
\int_{0}^{\tau}\int_{0}^{1}\alpha_{i}(t)\delta(x-x_{i})v(x,t)\text{d}x\text{d}t.
\end{split}
\end{equation}
\end{mydef}

\section{In-domain Actuation Design via Boundary Control}\label{Sec: Mapping}
\subsection{Relating In-domain Actuation to Boundary Control}\label{Sec: map to the boundary}
In the scheme developed in the present work, we use $N$ feedforward control signals to deform the beam, while dedicating the feedback stabilizing control to one actuator. For notational simplicity, we assign the feedback control to the $(N+1)^{\text{th}}$ actuator. Hence, the dynamic of the desired trajectory driven by the feedforward control can be expressed as
\begin{subequations}\label{indomain_form ff}
\begin{align}
  &w^{d}_{tt}(x,t)+ w^{d}_{\mathit{xxxx}}(x,t) = \sum_{j=1}^{N} \alpha_j(t)\delta(x-x_j),\;x\in \left(0,1\right), t > 0,\label{indomain_form1 ff} \\
  &w^{d}(0,t)=w^{d}_{x}(1,t)=w^{d}_{\mathit{xx}}(0,t)=w^{d}_{\mathit{xxx}}(1,t) = 0, \; t> 0,\label{BC1 ff}\\
  &w^{d}(x,0) = w_{t}^{d}(x,0)=0, \;x\in \left(0,1\right). \label{IV ff}
\end{align}
\end{subequations}
A weak solution of \eqref{indomain_form ff} can be defined in a similar way as that of \eqref{indomain_form} described in Definition~\ref{Def:weak_solution}. Note that it is shown later that the initial condition given in \eqref{IV} will be captured by the regulation error dynamics. Therefore, feedforward control design can be carried out based on System~\eqref{indomain_form ff} with zero-initial conditions.

Due to the fact that the model given in (\ref{indomain_form ff}) is driven by unbounded inputs, we will apply a sequence of cutoff functions $\varphi(x-x_j)$
having unit mass and satisfying $\int_{-\infty}^{+\infty}(x-x_j)\varphi(x-x_j)\text{d}x = 0$, also called \emph{blobs} \cite{Abramowitz1972},
to approximate $\delta(x-x_{j})$ in the sense of distributions. We will then show that in steady state, $\overline{w}^d$ can be approximated by a sufficiently smooth function. Specifically, we consider first the following boundary controlled PDE:
\begin{subequations}\label{target}
\begin{align}
  &u_{tt}(x,t)+ u_{\mathit{xxxx}}(x,t) = 0, \;\; x\in \left(0,1\right),\ t > 0,\label{target1} \\
  &u(0,t)=u_{x}(1,t)=u_{\mathit{xx}}(0,t)=0, \;\; t>0,\label{BC2}\\
  &u_{\mathit{xxx}}(1,t) = g(t), \;\; t>0, \label{BC3}\\
  &u(x,0) = u_{t}(x,0)= 0,\;\; x\in \left(0,1\right), \label{IV2}
\end{align}
\end{subequations}
where $g(t)=\sum_{j=1}^Ng_{j}(t)$. Throughout this paper, without special statements, we assume that $g_{j}(t)\in C^{3}(\left[0,+\infty\right))$ and $g_j(0)=\dot{g}_j(0)=0$ for $j=1,2,...,N$. Notice that the motivation behind considering the system of the form given in \eqref{target} is that it allows employing the techniques of boundary control for feedforward control design while avoiding early truncations of dynamic model and/or controller structure.

A weak solution to \eqref{target} can be defined in a similar way as that of \eqref{indomain_form} described in Definition~\ref{Def:weak_solution}.
\begin{mydef}
Let $T>0$. A weak solution to the problem~\eqref{target} is a function $u\in
C(\left[0,T\right]; \Phi)\cap C^1(\left[0,T\right]; L^2(0,1))$ satisfying
\begin{align}
  u(x,0) = u_{t}(x,0)=0,\ \ \ \ x\in \left(0,1\right),\notag
\end{align}
such that, for every $v\in C^{1}(\left[0,T\right];\Phi)$, one has for almost every $\tau\in \left[0,T\right]$
\begin{align}\label{eq: weak solution w2}
&\int_{0}^{1}u_{t}(x,\tau)v(x,\tau)\text{d}x-\int_{0}^{\tau}
\int_{0}^{1}u_{t}(x,t)v_{t}(x,t)\text{d}x\text{d}t\\
&+\int_{0}^{\tau}g(t)v(1,t)\text{d}t+\int_{0}^{\tau}\int_{0}^{1}
u_{\mathit{xx}}(x,t)v_{\mathit{xx}}(x,t)\text{d}x\text{d}t=0.\notag
\end{align}
\end{mydef}
Let $\psi(x,t)=u(x,t)-\sum_{j=1}^{N}g_{j}(t)H_{j}(x)$, where $H_{j}(x)$, $j=1,2,...,N$, are defined in $\left[0,1\right]$ satisfying
\begin{subequations}
\begin{align}
 &H_{\mathit{jxxxx}}(x)=\varphi(x-x_j), \label{eq: H}\\
 &H_{j}(0)=H_{jx}(1)=H_{\mathit{jxx}}(0)=0, H_{\mathit{jxxx}}(1) = 1, \label{Hj}
\end{align}
\end{subequations}
with $\varphi(x)\in L^2(\mathbb{R})$.
Then by lifting, \eqref{target} can be transformed into the following PDE with zero boundary conditions
\begin{subequations}\label{new problem}
\begin{align}
  &\psi_{tt}(x,t)+ \psi_{\mathit{xxxx}}(x,t) \nonumber \\
= & -\sum_{j=1}^{N}\ddot{g}_{j}(t)H_j(x)-\sum_{j=1}^{N}g_j(t)H_{\mathit{jxxxx}}(x), \;  x\in \left(0,1\right), t > 0, \label{lifting a} \\
  &\psi(0,t)=\psi_{x}(1,t)=\psi_{\mathit{xx}}(0,t)=\psi_{\mathit{xxx}}(1,t) = 0,\\
  &\psi(x,0) =\psi_{t}(x,0)=0, \;\;x\in\left(0,1\right).
\end{align}
\end{subequations}
In order to relate \eqref{indomain_form ff} to
\eqref{target}, we first establish a relationship between \eqref{indomain_form ff} and
\eqref{new problem}, especially in steady state. Let $\alpha_j(t) = -g_j(t)$ and suppose that $\lim_{t\rightarrow \infty}\alpha_j(t)= \overline{\alpha}_j$ and $\lim_{t\rightarrow \infty}g_j(t)= \overline{g}_j$ for $j = 1,\ldots, N$. We have then in steady state:
\begin{subequations}\label{lifting 2}
\begin{align}
\overline{\psi}_{\mathit{xxxx}}(x) &= -\sum_{j=1}^{N}\overline{g}_j H_{\mathit{jxxxx}}(x) \nonumber\\
                           &=  \sum_{j=1}^{N} \overline{\alpha}_j\varphi(x-x_j),\label{lifting 2 a}\; x\in \left(0,1\right), t > 0, \\
\overline{\psi}(0)&=\overline{\psi}_{x}(1)=\overline{\psi}_{\mathit{xx}}(0)=\overline{\psi}_{\mathit{xxx}}(1) = 0.
\end{align}
\end{subequations}
For $j=1,2,...,N$, given a sequence of blobs $\{\varphi_{m}(x-x_j)\}$, we seek a sequence of functions $\{ H_{j}^m(x)\}$ such that
\begin{equation}\label{H varphi}
  H^{m}_{\mathit{jxxxx}}(x)=\varphi_{m}(x-x_j)
\end{equation}
with $H^m_{j}(x)$ satisfying \eqref{Hj} and $\varphi_{m}(x-x_j)\rightarrow \delta(x-x_{j})$ in $\mathscr{D}'(0,1)$ as $m \rightarrow +\infty$. Therefore, considering \eqref{lifting 2} and the steady-state model of
\eqref{indomain_form ff}, we have
$\overline{\psi}_{m}\rightarrow\overline{w}^d$ in $C^1(\left[0,1\right])$ as $m
\rightarrow +\infty$ (see Theorem~\ref{proposition 3}). Hence, for feedforward control design, we may consider the systems \eqref{new
problem} and \eqref{lifting 2}.
\begin{mylemma} \label{proposition 1} \cite{Abramowitz1972}
Let $\varphi_{m}(x)\in L^2(\mathbb{R})$ be defined by
\begin{equation}\label{Eq: sinc}
\varphi_{m}(x)=\frac{1}{\pi}\frac{\sin mx}{x}.
\end{equation}
Then $\varphi_{m}(x)$ has the following properties:
\begin{enumerate}[(i)]
  \item $\displaystyle\int_{-\infty}^{+\infty}\varphi_{m}(x)\text{d}x
         =\displaystyle\int_{-\infty}^{+\infty}\frac{1}{\pi}\frac{\sin mx}{x}\text{d}x=1$.
  \item $\varphi_{m}(x)\rightarrow \delta(x)$ in $\mathscr{D}'(\mathbb{R})$ as $m \rightarrow \infty$.
\end{enumerate}
\end{mylemma}
By taking $\varphi_{m}$ given in \eqref{Eq: sinc} as the input to \eqref{H varphi}, we get
\begin{align}\label{Eq: Hjm}
  H_{j}^{m}(x)=&\frac{1}{6}x^3-\frac{1}{2}x \nonumber \\
               &+\int_{0}^{x}\int_{1}^{z}\int_{0}^{y}\int_{1}^{t}
                \varphi_{m}(s-x_{j})\text{d}s\text{d}t\text{d}y\text{d}z.\notag
\end{align}
\begin{mytheorem}\label{proposition 3}
Let $\varphi_{m}(x)$ be defined as in Lemma~\ref{proposition
1}. Assume that $\alpha_j(t)=-g_j(t)$ tends to $\overline{\alpha}_j = -\overline{g}_j $ as $t\rightarrow \infty$ for all $j=1,2,...,N$.
Denote by $\overline{\psi}_{j}^{m}$ and by $\overline{w}^{d}_j $ the
steady-state solutions of System~\eqref{indomain_form ff} and System~\eqref{new
problem}, respectively. Then
$\overline{\psi}_{j}^{m}\rightarrow\overline{w}^{d}_j $ in
$C^1(\left[0,1\right])$ as $m\rightarrow +\infty$ for $j=1,2,...,N$.
\end{mytheorem}
\begin{proof}
In the steady state, we have
\begin{subequations}
\begin{align}
&\overline{\psi}^{m}_{\mathit{jxxxx}}(x) = -\overline{g}_jH^{m}_{\mathit{jxxxx}}(x)=
\overline{\alpha}_j\varphi_{m}(x-x_j),
    \\
  &\overline{\psi}_{j}^{m}(0)=\overline{\psi}_{jx}^{m}(1)
  =\overline{\psi}^{m}_{\mathit{jxx}}(0)=\overline{\psi}^{m}_{\mathit{jxxx}}(1) = 0,
\end{align}
\end{subequations}
and
\begin{subequations}\label{steady beam equation}
\begin{align}
&\overline{w}^{d}_{\mathit{jxxxx}}(x) =  \overline{\alpha}_j\delta(x-x_{j}),
    \\
  &\overline{w}^{d}_j(0)=\overline{w}^{d}_{jx}(1)
  =\overline{w}^{d}_{\mathit{jxx}}(0)=\overline{w}^{d}_{\mathit{jxxx}}(1) = 0.
  \end{align}
\end{subequations}
Taking $v(x)\in \mathscr{D}(0,1)$, with $v(0)=v_x(1)=v_{\mathit{xx}}(0)=0$, as a test function and integrating by parts, we get
\begin{align*}
&\int_{0}^{1}\left(\overline{\psi}^m_{jx}(x)-\overline{w}^{d}_{jx}(x)\right)v_{\mathit{xxx}}(x)\text{d}x \\
=&\overline{\alpha}_j\int_{0}^{1}\left(\varphi_{m}(x-x_j)-\delta(x-x_{j})\right)v(x)\text{d}x.
\end{align*}
Since $\varphi_{m}(x-x_j)\rightarrow\delta(x-x_{j})$ in the sense of
distributions as $m\rightarrow +\infty$ and $v_{\mathit{xxx}} \in
\mathscr{D}(0,1)$, it follows that
$\overline{\psi}^m_{jx}\rightarrow\overline{w}^{d}_{jx}$ in the
sense of distributions as $m\rightarrow +\infty$ for $j=1,2,...,N$.
Furthermore, as $\overline{\psi}^m_{jx}\in L^1(0,1)$ and
$\overline{w}^{d}_{jx} \in L^1(0,1)$, we have
$\overline{\psi}^m_{jx}\rightarrow\overline{w}^{d}_{jx} $ a.e. in
$(0,1)$ (see Lemma 3.31 of \cite{AF2003}, page 74). Then by
the continuity of $\overline{\psi}^m_{jx}$ and
$\overline{w}^{d}_{jx}$, we conclude that
$\overline{\psi}^m_{jx}\rightarrow\overline{w}^{d}_{jx}$ pointwise
in $(0,1)$. Therefore
$\overline{\psi}^m_{j}\rightarrow\overline{w}^{d}_j $ in
$C^1(\left[0,1\right])$ as $m\rightarrow +\infty$ for $j=1,2,...,N$.
\end{proof}

\begin{remark}\label{RM: th2}
Note that describing the control acting on discrete points in the domain by Dirac delta functions is a mathematical abstraction, which allows for the beam equation to possess some fundamental properties, in particular the exact controllability (see, e.g., \cite{jacob1999equivalent}), and facilitates control design. However, the real physical devices cannot generate unbounded control concentrated on a point. From this viewpoint, Theorem~\ref{proposition 3} provides an assessment of the gap between the performance achievable by utilizing real physical devices and that predicted by the corresponding idealized mathematical model.
\end{remark}

\subsection{Well-posedness of Cauchy Problems}\label{Sec: Well-posedness}
Well-posedness analysis is essential to the approach
developed in this work. In this subsection, we establish the existence and the uniqueness of weak
solutions of equations \eqref{indomain_form}, \eqref{indomain_form ff}, \eqref{target} and \eqref{new problem}.

\begin{mytheorem}\label{the: wellposedness}
The following statements hold true:
\begin{enumerate}[(i)]
  \item Assume $\alpha_{j}\in L^{2}(0,T)$ for $j=1,2,...,N+1$. Let $h_{0}\in \Phi,\ h_{1} \in L^{2}(0,1)$ and $T>0$. Then System \eqref{indomain_form} and System \eqref{indomain_form ff} has a unique weak solution $w\in
C(\left[0,T\right]; \Phi)\cap C^1(\left[0,T\right]; L^2(0,1))$ and $w^d\in
C(\left[0,T\right]; \Phi)\cap C^1(\left[0,T\right]; L^2(0,1))$, respectively.
  \item Let $T>0$ and $g\in C^{2} (\left[0,T\right])$. Then System \eqref{target} has a unique weak solution $u\in
C(\left[0,T\right]; \Phi)\cap C^1(\left[0,T\right]; L^2(0,1))$. Furthermore, if $g\in
C^{3} (\left[0,T\right])$ and $H_{j}^{m}(x)$, $j=1,2,...,N$, is defined as in
\eqref{Hj}, then System \eqref{new problem} has a unique solution
$\psi_{m} \in C(\left[0,T\right]; \Phi)\cap C^1(\left[0,T\right]; L^2(0,1))$.
\end{enumerate}
\end{mytheorem}
\begin{proof}
The proof of (i) can be proceeded step by step as in Proposition~3.1 of \cite{ammari:2000}. We prove the first result of (ii) and the second part can be proceeded in the same way. We assume first $g(t)\in C^3(\left[0,T\right])$ and consider the following system:
\begin{subequations}\label{10abc}
\begin{align}
 &\nu_{tt}(x,t)+ \nu_{\mathit{xxxx}}(x,t) =\left(\dfrac{1}{2}x-\dfrac{1}{6}x^{3}\right)\ddot{g}(t),
 \; x\in\Omega, t > 0,\\
 &\nu(0,t)=\nu_{x}(1,t)=\nu_{\mathit{xx}}(0,t)=\nu_{\mathit{xxx}}(1,t) = 0,\\
 &\nu(x,0) = 0, \nu_{t}(x,0)=0, \;\;x\in\Omega.
\end{align}
\end{subequations}
Let $X=\Phi\times L^{2}(0,1)$ and $H^4_{(0)}(0,1)=\{u\in H^4 (0,1);
u(0)=u_{x}(1)=u_{\mathit{xx}}(0)=u_{\mathit{xxx}}(1)=0\}$. Define the inner product on $X$
by $\langle(u_{1},v_{1}),(u_{2},v_{2})\rangle_{X}=\int_{0}^{1}(
u_{\mathit{1xx}}u_{\mathit{2xx}}+v_{1}v_{2})\text{d}x$. Define the subspace $\mathcal
{D}(A_0)\subset X$ by
 $\mathcal {D}(A_0)=\{(u,v); \ (u,v) \in H^{4}_{(0)}(0,1)\times \Phi\}$,
with the corresponding operator $A_{0}$:~$\mathcal {D}(A_0)\rightarrow X$
defined as:
\begin{equation*}
A_0 \binom{u}{v} = \binom{v}{-u_{\mathit{xxxx}}}.
\end{equation*}
One may easily check that
$\mathcal {D}(A_0)$ is dense in $X$, $A_0$ is closed, $A^{*}_{0}=-A_0$, and $\langle A_0z,z\rangle_{X}=0$.
Thus, by Stone's theorem, $A_0$ generates a semigroup of isometries on $X$. Based on a classical result on perturbations
of linear evolution equations (see, e.g., Theorem~1.5, Chapter~6, page~187, \cite{pazy1983semigroups}) there exists uniquely
$
z=(z_1,z_2)\in C^{1}(\left[0,T\right]; X)\cap C(\left[0,T\right];\mathcal
{D}(A_0)),
$ 
such that
\begin{equation*}\label{z}
\frac{\text{d}z}{\text{d}t}=A_0z+
\begin{pmatrix}
0\\
\left(\dfrac{1}{2}x-\dfrac{1}{6}x^{3}\right)\ddot{g}(t)
\end{pmatrix},
\ z(\cdot, 0)=(0,\ 0),
\end{equation*}
which implies that \eqref{10abc} has a unique solution $\nu=z_1 \in C(\left[0,T\right]; H^4_{(0)}(0,1))\cap C^1(\left[0,T\right]; \Phi) $ in the usual sense. Particularly, $\nu \in C(\left[0,T\right]; \Phi)\cap C^1(\left[0,T\right]; L^2(0,1)) $ is a weak solution. A
direct computation shows that $u=\nu-
\left(\frac{1}{2}x-\frac{1}{6}x^{3}\right)g(t) $ is a solution
of \eqref{target} in the usual sense and, in particular, it is a weak solution.

Now for $g\in C^{1}(\left[0,T\right])$, let $g_{n}\in C^{3}(\left[0,T\right])$
such that $g_{n}\rightarrow g \ \text{in}\ C^{1}(\left[0,T\right])$. Consider $ u_n=\nu_n-
\left(\frac{1}{2}x-\frac{1}{6}x^{3}\right)g_n(t)$, where $\nu_{n}$ is the solution of \eqref{10abc} corresponding to the data $g_{n}(t)$. Then
arguing as above and taking limit as in Chapter~2 of \cite{coron2007}, we can obtain the existence
and the uniqueness of a weak solution of \eqref{target}.
\end{proof}

\section{Feedback Control and Stability of the Inhomogeneous System}\label{Sec: Stability}
The validity of the proposed scheme requires a suitable closed-loop control that can guarantee the stability of the original inhomogeneous system.
As the $(N+1)^{\text{th}}$ actuator is dedicated to stabilizing control, \eqref{indomain_form1} can be written as
\begin{align}\label{indomain_form N}
   &w_{tt}(x,t)+ w_{\mathit{xxxx}}(x,t) - \alpha_{N+1}(t)\delta(x-x_{N+1}) \nonumber \\
  =& \sum_{j=1}^{N} \alpha_j(t)\delta(x-x_j),
   \; x\in (0,1), t > 0.
\end{align}
Suppose further that the feedback control is taken as \cite{ammari:2000,Delfour:1987}
\begin{equation}\label{Eq: feedback}
  \alpha_{N+1}(t) = -kw_t\left(x_{N+1},t\right),
\end{equation}
where $k$ is a positive-valued constant. Then in closed-loop,
\eqref{indomain_form N} becomes
\begin{align}\label{indomain_form N CL}
    &w_{tt}(x,t) + w_{\mathit{xxxx}}(x,t) + k w_t\left(x,t\right)\delta(x-x_{N+1}) \nonumber\\
  = &\sum_{j=1}^{N} \alpha_j(t)\delta(x-x_j),
   \; x\in (0,1), t > 0.
\end{align}
Let $\mathcal{D}(A)=\{(w,v); (w,v)\in (H^{2}(0,1)\cap (H^{4}(0,x_{1})\cup H^{4}(x_{1},x_{2})\cup\cdots \cup H^{4}(x_{N},x_{N+1})\cup H^{4}(x_{N+1},1))) \times
H^{2}(0,1), w(0)=w_{x}(1)=w_{\mathit{xx}}(0)=w_{\mathit{xxx}}(1)=0,
v(0)=v_{x}(1)=0\}$ and $ X$ be defined as in the proof of Theorem~\ref{the: wellposedness}.
Let $T>0$ and $\alpha_{j}\in L^{2}(0,T)$ for $j=1,...,N$. Assume $ ( h_{0},
h_{1})\in X$. To address the stability of System
\eqref{indomain_form N CL} with the boundary conditions \eqref{BC1}
and initial conditions \eqref{IV}, we consider the corresponding
linear control system under the following abstract form:
\begin{subequations}\label{eq: indomain ABS}
\begin{align}
  &\dot{z}= Az +B\alpha,\;\; t > 0,\label{eq: indomain ABS dyn} \\
     &z(0)= z^0=( h_{0},
h_{1})^T, \label{eq: indomain ABS IC}
\end{align}
\end{subequations}
where $z = (w,\; v)^T$, $\alpha =(\alpha_1, \; \cdots, \;\alpha_N )^T $, $A$:~$\mathcal{D}(A) \rightarrow X$ is defined as:
\begin{align}\label{Eq: op A}
  A\binom {w}{v} = \binom {v}{- w_{\mathit{xxxx}} -k v\delta(x-x_{N+1})},
\end{align}
with $v=w_t$ and $B$:~$\mathbb{R}^{N}\rightarrow \mathcal{D}'(A^*)$, where $A^*$ is the adjoint of $A$,
which is defined as:
\begin{equation}\label{Eq: op B}
    (B\alpha)y= \binom {0}{\sum_{j=1}^{N}\alpha_j\delta(x-x_j)}^Ty,
      \;\;\forall\  y\in \mathcal{D}(A^*).
\end{equation}
Note that $B^*$:~$\mathcal{D}(A^*)\rightarrow\mathbb{R}^{N}$ is defined by
\begin{equation}\label{Eq: op B*}
   B^*y=(y_2(x_{1}),\ldots,y_2(x_{N}))^{T}, \forall y=\binom {y_1}{y_2}\in \mathcal{D}(A^*).
\end{equation}
Let
$U=\mathbb{R}^{N}$. Then the solution of the Cauchy problem
\eqref{eq: indomain ABS} can be defined as follows (see
\cite{coron2007}, Definition 2.36, p53):
\begin{equation}\label{Eq: weak_ansewr1}
\begin{split}
     &\langle z(\tau),y^{\tau}\rangle_X \\
    =&\langle  z^0,S^*(\tau)y^{\tau}\rangle_X
      + \int_0^{\tau} \!\!\langle  \alpha(t),B^*S^*(\tau-t)y^{\tau}\rangle_{U}\text{d}t,\\
     &\forall\ \tau \in \left[0,T\right], \forall\  y^{\tau} \in X.
\end{split}
\end{equation}
One may verify, as in Chapter~2 of \cite{coron2007}, that the
solution defined by \eqref{Eq: weak_ansewr1} is also a weak solution
of \eqref{indomain_form N CL} under the form given in
Definition~\ref{Def:weak_solution}.

\begin{mytheorem}\label{the: stability}
Assume $\alpha_{j} \in L^\infty(0,+\infty)$, $j=1,\ldots,N$. For any $x_{N+1} \in \left(0,1\right]$ being a rational number with coprime factorization, there exist positive constants $C_1$, $C_2$ and $\lambda$ such that for any $0 \leq \tau \leq T < \infty$, there holds:
\begin{align}\label{Eq: exponential stability 3}
  \|z(\tau)\|_{X} &\leq  C_1e^{-\lambda \tau}\|z^0\|_{X} + C_2\|\alpha\|_{L^\infty(0,T)}.
\end{align}
\end{mytheorem}
\begin{proof}
The proof is a slight modification of Theorem~2.37 in \cite{coron2007}. First, the admissible property of $B$ can be obtained as (3.38) in \cite{ammari:2000}. A direct computation gives
\begin{align*}
\bigg\langle A\begin{pmatrix}
u\\
v
\end{pmatrix},\begin{pmatrix}
u\\
v
\end{pmatrix}\bigg\rangle_{X}=-k|v(x_{N+1})|^2,\ \ \ \ \forall \begin{pmatrix}
u\\
v
\end{pmatrix}\in\mathcal {D}(A).
\end{align*}
Therefore, $A$ is a dissipative operator. Moreover, it can be directly verified that $A$ is onto and hence, according to Theorems 4.3 and 4.6 of \cite{Pederson2000} (p.~14-15), it generates a $C_0$-semigroup of linear contractions $S(t)$ acting on $X$. Furthermore, $S(t)$ is
exponentially stable if $x_{N+1} \in (0,1)$ is a rational number with coprime factorization, in particular for $x_{N+1} = 1$ (see, e.g., \cite{ammari:2000,Delfour:1987}), i.e. there exist two positive constants $C_1$ and $\lambda$, such that
\begin{align*}
  \|S(t)\|_{\mathcal {L}(X;X)} &\leq  C_1e^{-\lambda t},\ \forall\  t\geq 0.
\end{align*}
Then for a weak solution defined in \eqref{Eq: weak_ansewr1}, we have that for all $0 \leq \tau \leq T < \infty$:
\begin{align*}
      &\langle z(\tau),y^{\tau}\rangle_X \\
    =& \langle  S(\tau)z^0,y^{\tau}\rangle_X + \int_0^{\tau} \!\!\langle  S(\tau-t)B\alpha(t),y^{\tau}\rangle_{U}\text{d}t\\
    =& \langle  S(\tau)z^0,y^{\tau}\rangle_X + \;\left\langle  \int_0^{\tau} S(\tau-t)B\alpha(t)\text{d}t,y^{\tau}\right\rangle_{U}\\
    \leq& C_1e^{-\lambda\tau}\|z^0\|_{X} \|y^{\tau}\|_{X}
         +C_1\|\alpha\|_{L^{\infty}(0,T)} \|y^{\tau}\|_{L^{\infty}(0,T)}\int_0^{\tau} e^{-\lambda(\tau-t)}\text{d}t \\
    \leq& C_1e^{-\lambda\tau}\|z^0\|_{X} \|y^{\tau}\|_{X} + \dfrac{C_1}{\lambda}\|\alpha\|_{L^{\infty}(0,T)} \|y^{\tau}\|_{L^{\infty}} \\
    \leq& C_1e^{-\lambda\tau}\|z^0\|_{X}\|y^{\tau}\|_{X} + \dfrac{C_1}{\lambda}\|\alpha\|_{L^{\infty}(0,T)}\|y^{\tau}\|_{H^1}.
\end{align*}
Applying the interpolation formula
\begin{equation*}
\|u_{x}\|_{L^2}\leq c\left( \|u\|_{L^2}  + \|u_{\mathit{xx}}\|_{L^2}\right),
\end{equation*}
where $c$ is a positive constant, and the Sobolev embedding theorem yields
\begin{equation*}
   \|y^{\tau}\|_{H^1}\leq c\|y^{\tau}\|_{H^2}\leq c\|y^{\tau}\|_{X}.
\end{equation*}
We have then
\begin{equation*}
\langle z(\tau),y^{\tau}\rangle_X \leq \left(C_1e^{-\lambda\tau}\|z^0\|_{X}
                                           +\dfrac{C_1c}{\lambda}\|\alpha\|_{L^{\infty}(0,T)} \right)\|y^{\tau}\|_{X},
\end{equation*}
which implies that \eqref{Eq: exponential stability 3} holds with $C_2 = C_1c/\lambda$.
\end{proof}

Now consider the regulation error defined as $e(x,t) = w(x,t) - \overline{w}^{d}(x)$. Denoting $\Delta\alpha_j(x,t) = \alpha_j(x,t) - \overline{\alpha}_j(x)$, $j = 1, \ldots, N$, then from \eqref{indomain_form N CL}, \eqref{BC1}, \eqref{IV} and the steady-state model of \eqref{indomain_form ff}, the regulation error dynamics satisfy
\begin{subequations}\label{target_error}
\begin{align}
   &e_{tt}(x,t) + e_{\mathit{xxxx}}(x,t) + k e_t\left(x,t\right)\delta(x-x_{N+1}) \notag\\
  =& \sum_{j=1}^{N} \Delta\alpha_j(t)\delta(x-x_j),
   \; x\in(0,1), t > 0, \label{target1_error} \\
  &e(0,t)=e_{x}(1,t)=e_{\mathit{xx}}(0,t)=e_{\mathit{xxx}}(1,t) = 0, \;\; t>0, \label{BC3_error}\\
  &e(x,0) = e_0(x) = h_0(x)-\overline{w}^d(x),\notag\\
  &e_{t}(x,0)=e_1(x) = h_1(x), \;\;\;\;\;\;\;\;\;\;\;\; x\in (0,1). \label{IV2_error}
\end{align}
\end{subequations}

Obviously, the regulation error dynamics are in an identical form
as \eqref{indomain_form N CL} with the same type of boundary conditions.
We can then consider the solution of System~\eqref{target_error}
defined in the same form given by \eqref{Eq: weak_ansewr1}
with $z_e = (e, e_t)$ and $z^0_e = (e_0, e_1)$.
\begin{mycorollary}\label{corollary error convegency}
Assume that all the conditions in Theorem~\ref{the: stability} are
fulfilled and $ z^0_e \in X$. Then there exist positive constants
$C_1$, $C_2$ and $\lambda$, independent of $t$, such that for any
$t\geq 0$, there holds:
\begin{align}\label{Eq: exponential stability 4}
        \|z_e(t)\|_{X} &\leq  C_1e^{-\lambda t}\|z^0_e\|_{X}
        +
        C_2\|\Delta\alpha\|_{L^\infty(0,+\infty)}.
\end{align}
Moreover, if $\lim_{t\rightarrow \infty}\Delta\alpha_j = 0$ for all $j = 1,\ldots, N$, then $\lim_{t\rightarrow \infty}e(x,t) = 0$, $\forall x\in (0,1)$.
\end{mycorollary}

The feedforward control satisfying the conditions for closed-loop stability and regulation error convergence can be obtained through motion planning, as presented in the next section.

\section{Motion Planning and Feedforward Control}\label{Sec: Controller design}

According to the principle of superposition for linear systems, we consider in feedforward control design the dynamics of System~\eqref{target} corresponding to the input $g_j(t)$ described in the following form:
\begin{subequations}\label{sys_tracking}
\begin{align}
  &u_{jtt}(x,t)+ u_{\mathit{jxxxx}}(x,t) = 0,\label{sys_tracking_main}\\
  &u_{j}(0,t)=u_{jx}(1,t)=u_{\mathit{jxx}}(0,t)=0,&\label{bc_2}\\
  &u_{\mathit{jxxx}}(1,t) = g_j(t),& \label{in_j}\\
  &u_{j}(x,0) =u_{jt}(x,0) =0.
\end{align}
\end{subequations}
The required control signal $g_j(t)$ should be designed so that the output of System~\eqref{sys_tracking} follows a prescribed function $u_j^d(x_j,t)$, while guaranteeing that the conditions in Theorem~\ref{the: stability} and Corollary~\ref{corollary error convegency} are fulfilled.

Motion planning amounts then to taking $u_j^d(x_j,t)$ as the desired output and to generating the full-state trajectory of the subsystem $u_j(x,t)$.
By applying a standard Laplace transform-based procedure (see, e.g., \cite{Meurer:2013,Aoustin1997,Fliess:1999-CDC,Rudolph2003}), we can obtain the full-state trajectory with zero initial values expressed in terms of the so-called flat output, $y_j(t)$, and its time-derivatives:
\begin{equation}\label{+16}
\begin{split}
 u_{j}(x,t)
=&\left(\frac{1}{2}x-\frac{1}{6}x^{3}\right)y_j(t)\\
 &+\sum _{n=1}^{\infty}\left(
   \sum_{k=0}^{n}\frac{x^{4k+1}}{(4k+1)!(4(n\!-\!k)+2)!}\right. \\
 &\left.-\sum_{k=0}^{n}\frac{x^{4k+3}}{(4k+3)!(4(n\!-\!k))!}
  \right)(-1)^{n}y_j^{(2n)}(t).
\end{split}
\end{equation}

Now let $y_j(t)= \overline{y}_j\phi_j(t)$, where $\phi_{j}(t)$ is a smooth function evolving from 0 to 1. The full-state trajectory can then be written as
\begin{equation}\label{+17}
\begin{split}
 u_{j}(x,t)
=&\overline{y}_j\left(\frac{1}{2}x-\frac{1}{6}x^{3}\right)\phi_j(t)\\
 &+\overline{y}_j\sum _{n=1}^{\infty}\left(
   \sum_{k=0}^{n}\frac{x^{4k+1}}{(4k+1)!(4(n\!-\!k)+2)!} \right.\\
 &\left. -\sum_{k=0}^{n}\frac{x^{4k+3}}{(4k+3)!(4(n\!-\!k))!}
  \right)(-1)^{n}\phi_j^{(2n)}(t).
\end{split}
\end{equation}

The corresponding input can be computed from \eqref{in_j}, which yields
\begin{equation}\label{+18}
\begin{split}
g_{j}(t)=&-\overline{y}_j\phi_j(t)+\overline{y}_j\sum_{n=1}^{\infty}
         \left(\sum_{k=1}^{n}\frac{1}{(4k-2)!(4(n\!-\!k)+2)!} \right.\\
         &\left.-\sum_{k=0}^{n}\frac{1}{(4k)!(4(n\!-\!k))!}
        \right)(-1)^{n}\phi_j^{(2n)}(t).
\end{split}
\end{equation}

For set-point control, we need an appropriate class of trajectories enabling a rest-to-rest
evolution of the system. A convenient choice of $\phi_j(t)$ is the following smooth function:
\begin{equation}\label{Gevrey function}
\phi_{j}(t) = \left\{
\begin{array}{l l}
  0, & \quad \mbox{if $t \leq 0$}\\
  \dfrac{\displaystyle\int_0^t \exp(-1/(\tau(1-\tau)))^{\varepsilon}\text{d}\tau}{\displaystyle\int_0^T \exp(-1/(\tau(1-\tau)))^{\varepsilon}\text{d}\tau}, & \quad \mbox{if $t\in (0,T)$}\\
  1, & \quad \mbox{if $t \geq T$}\\
  \end{array} \right.
\end{equation}
which is known as Gevrey function of order $\sigma=1+1/\varepsilon$,
$\varepsilon>0$ (see, e.g., \cite{Rudolph2003}).

For the convergence of \eqref{+17} and \eqref{+18}, we have
\begin{myproposition}\label{proposition 2}
If $\phi_j(t)$ in the basic output $y^j(t) = \overline{y}^j\phi_j(t)$ is chosen
as a Gevrey function of order $1<\sigma<2$, then the infinite series
\eqref{+17} and \eqref{+18} are convergent.
\end{myproposition}

The proof of Proposition~\ref{proposition 2} is based on the bounds of Gevrey functions of order $\sigma$ given by (see, e.g., \cite{Dunbar:2003,Rodino:1993}):
\begin{equation}\label{gevrey}
    \exists K, M>0, \forall k \in \mathbb{Z}_{\geq 0}, \forall t\in\left[t_0,T\right], \left| \phi_{j}^{(k+1)}(t)\right|\leq M\frac{(k!)^{\sigma}}{K^k}.
\end{equation}
The detailed convergence analysis is given in Appendix~\ref{Appendix: Proof of Proposition 2}.

\begin{remark}
In general, the Gevrey bounds are unknown, but can be estimated
following the way presented in \cite{Dunbar:2003}. Furthermore, a
symmetric function in the transient phase can be considered to
improve convergency analysis \cite{Dunbar:2003}.
\end{remark}
Now let
\begin{subequations}
\begin{align*}
P_j(x)=&\frac{1}{2}x-\frac{1}{6}x^3,\notag\\
I_{j,m}(x)=&\int_{0}^{x}\int_{1}^{z}\int_{0}^{y}
\int_{1}^{t}\varphi_{m}(s-x_{j})\text{d}s\text{d}t\text{d}y\text{d}z,\notag\\
\Phi_{j,n}(x)=&\bigg(\sum_{k=0}^{n}\frac{x^{4k+1}}{(4k+1)!(4(n\!-\!k)+2)!}
               -\frac{x^{4k+3}}{(4k+3)!(4(n\!-\!k))!}\bigg)(-1)^{n},\notag\\
\Psi_{j,n}(1)=&\bigg(\sum_{k=1}^{n}\frac{1}{(4k-2)!(4(n\!-\!k)+2)!}
               -\sum_{k=0}^{n}\frac{1}{(4k)!(4(n\!-\!k))!}\bigg)(-1)^{n}.\notag
\end{align*}
\end{subequations}
Set $\sum_{j=1}^N\psi_{j}^{m}(x,t)=\psi_{m}(x,t), m>0$. By the definition of $\psi(x,t)$, we obtain
\begin{equation}\label{expansion of psi}
\begin{split}
\psi_{j}^{m}(x,t)=&\overline{y}_{j}I_{j,m}(x)\phi_{j}(t)
  +\overline{y}_{j}\sum_{n=1}^{\infty}\left( \Phi_{j,n}(x) \right.\\
&\left.+\Psi_{j,n}(1)P_{j}(x)-\Psi_{j,n}(1)I_{j,m}(x)\right)\phi_{j}^{(2n)}(t).
\end{split}
\end{equation}
To compute $I_{j,m}(x)$, we note that
\begin{align*}
\varphi_{m}(x-x_j)=\frac{1}{\pi}
\sum_{k=1}^{\infty}\frac{(-1)^{k+1}m^{2k-1}}{(2k-1)!}(x-x_{j})^{2k-2}.
\end{align*}
Therefore,
\begin{equation}\label{expansion of I}
\begin{split}
I_{j,m}(x)=&\frac{1}{\pi}
\sum_{k=1}^{\infty}\frac{(-1)^{k+1}m^{2k-1}}{(2k-1)}
\bigg(\frac{(x-x_{j})^{2k+2}}{(2k+2)!} \\
&-\frac{ (1-x_{j})^{2k-1}x^{3}}{6((2k-1)!)}
-\frac{x_{j}^{2k}x^2}{2(2k)!}
-\frac{(1-x_j)^{2k+1}x}{(2k+1)!}\\
& +
\frac{x_j^{2k}x}{(2k)!}
+\frac{(1-x_j)^{2k-1}x}{2(2k-1)!}
-\frac{x_j^{2k+2}}{(2k+2)!}
\bigg).
\end{split}
\end{equation}
\begin{myclaim}\label{Cl: claim 1}
$I_{j,m}(x)$ given in \eqref{expansion of I} is convergent for all $x, x_j \in (0,1)$ with respect to any fixed $m$.
\end{myclaim}
The proof of this claim is given in Appendix~\ref{Appendix: Proof of Claim 1}.

Based on Claim~\ref{Cl: claim 1}, the series in the right hand side of \eqref{expansion of psi} are convergent and $\psi_{j}^{m}(x,t)$ can be expanded by \eqref{expansion of psi} and \eqref{expansion of I}. Moreover, for $g_j(t)$ given in \eqref{+18}, $\psi_{j}^{m}(x,t)$ tends to
$
\overline{\psi}_{j}^{m}(x)=\overline{y}_{j}I_{j,m}(x)
$ as $t\rightarrow \infty$.
Note that for fixed $x$, the radius of convergence of $\overline{\psi}_{j,m}$ on $m$ is $\infty$. Thus, we can let $m\rightarrow +\infty$.

To complete the control design, we need to determine the amplitude of flat outputs $\overline{y}_j$, $j = 1, \ldots, N$, from the desired shape $\widetilde{w}^d(x)$ that may not necessarily be a solution of the steady-state beam equation \eqref{steady beam equation}. To this end, we propose to interpolate $\widetilde{w}^d(x)$ using the Green's function of \eqref{steady beam equation}, denoted by $G(x,\xi)$, which is the impulse response corresponding to the input $\delta(x-\xi)$ and is given by
\begin{equation*}\label{beam green's function j}
    G(x,\xi_j)=\left\{
               \begin{array}{ll}
               -\dfrac{x^3}{6}+x\xi_j\left(1-\dfrac{\xi_j}{2}\right) , & 0 \leq x < \xi_j; \\
              -\dfrac{\xi^3_j}{6}+\xi_j x\left(1-\dfrac{x}{2}\right) , & \xi_j \leq x \leq 1.
               \end{array}
             \right.
\end{equation*}
Due to the principle of superposition for linear systems, the solution to \eqref{steady beam equation}, $\overline{w}^d(x)$, can be expressed as
\begin{align*}\label{Eq: ss ref}
    \overline{w}^d(x) =& \int_0^1 \sum_{j=1}^N G(x,\xi)\overline{\alpha}_j \delta(x-\xi_j)\text{d}\xi \nonumber\\
                      =& \sum_{j=1}^{N} G(x,\xi_j)\overline{\alpha}_j.
\end{align*}
Now taking $N$ points on $\widetilde{w}^d(x)$ and letting $\overline{w}^d(x_j) = \widetilde{w}^d(x_i=j)$, $j = 1, \ldots, N$, yield
\begin{equation}\label{Eq: SS responses}
\begin{pmatrix}
   G(x_{1},\xi_{1}) & \ldots&
   G(x_{N},\xi_{1}) \\
   \vdots &\ddots & \vdots\\
   G(x_{1},\xi_{N}) & \ldots &
   G(x_{N},\xi_{N})\\
\end{pmatrix}
\begin{pmatrix}
\overline{\alpha}_{1}\\
\vdots\\
\overline{\alpha}_{N} \\
\end{pmatrix}
=\begin{pmatrix}
\widetilde{w}^d(x_1)\\
\vdots\\
\widetilde{w}^d(x_N) \\
\end{pmatrix},
\end{equation}
which represents a steady-state input to output map.
\begin{myclaim}\label{Cl: claim 2}
The map of \eqref{Eq: SS responses} is invertible for all $x_j, \xi_j \in (0, 1)$, $j = 1, \ldots, N$ and $x_i \neq x_j, \xi_i \neq \xi_j$, if $i \neq j$.
\end{myclaim}
The proof of this claim is given in Appendix~\ref{Appendix: Proof of Claim 2}.

As $\overline{\alpha}_j=-\overline{g}_j$ and $\lim_{t\rightarrow \infty} g_j(t) = \overline{g}_j = -\overline{y}_j$ for all $j = 1,\ldots, N$, we obtain from Claim~\ref{Cl: claim 2} that
\begin{equation}\label{Eq: amptidue of flat output}
\begin{pmatrix}
\overline{y}_{1}\\
\vdots\\
\overline{y}_{N} \\
\end{pmatrix}=
\begin{pmatrix}
   G(x_{1},\xi_{1}) & \ldots&
   G(x_{N},\xi_{1}) \\
   \vdots &\ddots & \vdots\\
   G(x_{1},\xi_{N}) & \ldots &
   G(x_{N},\xi_{N})\\
\end{pmatrix}^{-1}\!\!\!\!
\begin{pmatrix}
\widetilde{w}^d(x_1)\\
\vdots\\
\widetilde{w}^d(x_N) \\
\end{pmatrix}.
\end{equation}

\section{Simulation Study} \label{simulation}
In the simulation study, we consider the deformation control in which the desired shape is given by
\begin{equation}\label{Eq: desired shape}
\begin{split}
  \widetilde{w}^d(x) =& -10^{-3}\left(e^{-100\left(x-0.4\right)^2} + 2e^{-100\left(x-0.6\right)^2} \right.\\
                      &\left. +3e^{-400\left(x-0.7\right)^2}\right), \; x\in (0, 1),
\end{split}
\end{equation}
as shown in Fig.~\ref{No of actuators 1}. Note that the maximum amplitude of the shape given by \eqref{Eq: desired shape} is $3.8\times 10^{-3}$, which represents a typical micro-structure for which the beam length is of few centimeters and the displacement is of micrometer order.

In order to obtain an exponential closed-loop convergency, the actuator for feedback stabilization is located at the position $x_{N+1}=1$. To evaluate the effect of the number of actuators to interpolation accuracy, measured by $\|\widetilde{w}^d(x)-\overline{w}^d(x) \|_{L^1(0,1)}$ and control effort, we considered 3 setups with, respectively, 8, 12 and 16 actuators evenly distributed in the domain. It can be seen from Fig.~\ref{Fig: No of actuators} that the setup with 8~actuators exhibits an important interpolation error and the one with 16~actuators requires a high control effort in spite of a high interpolation accuracy. The setup with 12~actuators provides an appropriate trade-off between the interpolation accuracy and the required control effort, which is used in control algorithm validation.

A MATLAB Toolbox for dynamic Euler-Bernoulli beams simulation provided in Chapter~14 of \cite{Yang:2005} is used in numerical implementation. With this Toolbox, the simulation accuracy can be adjusted by choosing the number of modes used in implementation. In the simulation, we implement System~(\ref{indomain_form}) with initial conditions $h_0(x,0)=-3\times 10^{-3} e^{-400(x-0.8)^2}$ and $h_1(x,0)=0$.  As an undamped beam is unstable in open-loop, the controller tuning is started by determining a suitable value of the closed-loop control gain $k$. The basic outputs $\phi_j(t)$ used in the simulation are Gevrey functions of the same order. To meet the convergence condition given in Proposition~\ref{proposition 2}, the parameter in \eqref{Gevrey function} is set to $\varepsilon = 1.111$. The corresponding feedforward control signals with $T = 5$, $\alpha_1, \ldots, \alpha_{12}$, that steer the beam to deform are illustrated in Fig.~\ref{fig: controls}. The evolution of beam shapes and the regulation error are depicted in Fig.~\ref{Fig: beam deformation}. It can be seen that the beam is deformed to the desired shape and the regulation error vanishes along the entire beam, which confirms the expected performance of the developed control scheme.

\begin{figure}[h]\
  \centering
  \subfigure[]{\label{No of actuators 1}\includegraphics[scale=.5]{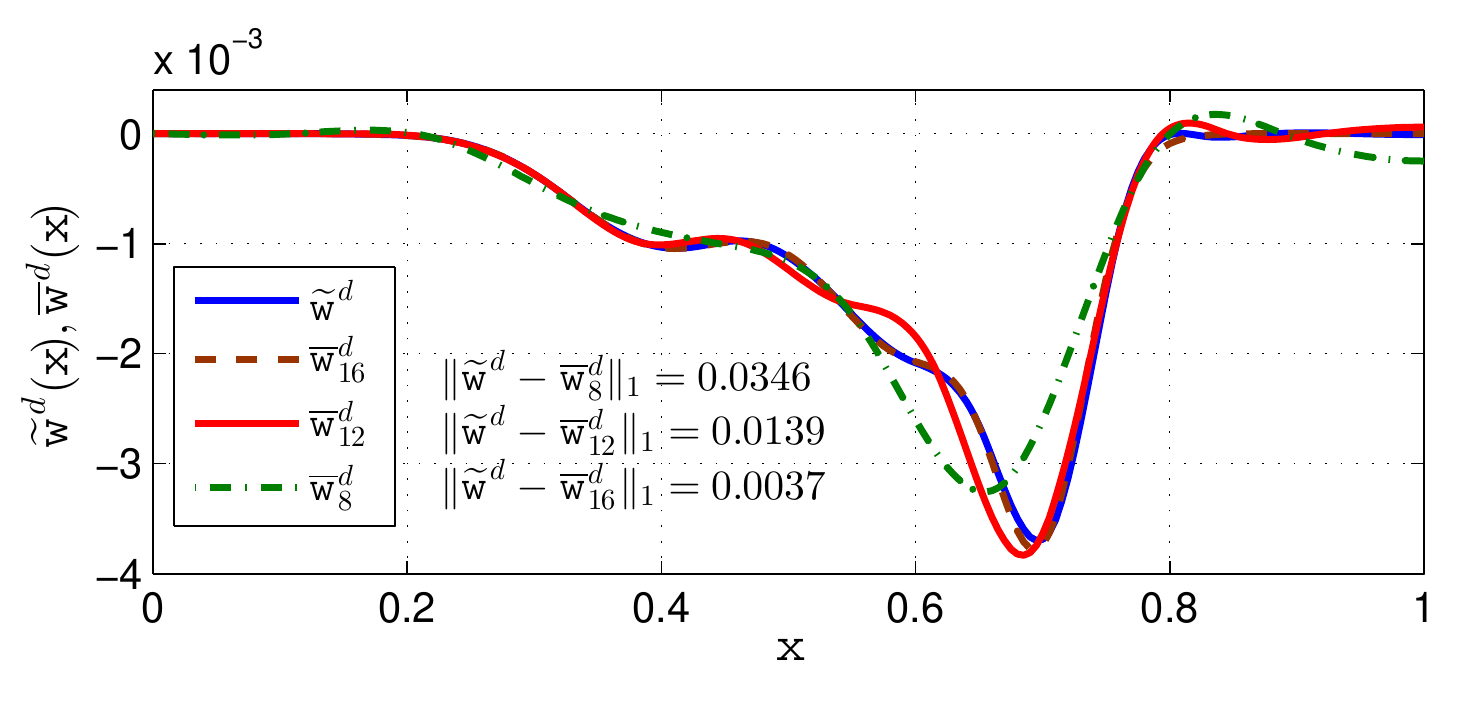}}
  \subfigure[]{\label{No of actuators 2}\includegraphics[scale=.5]{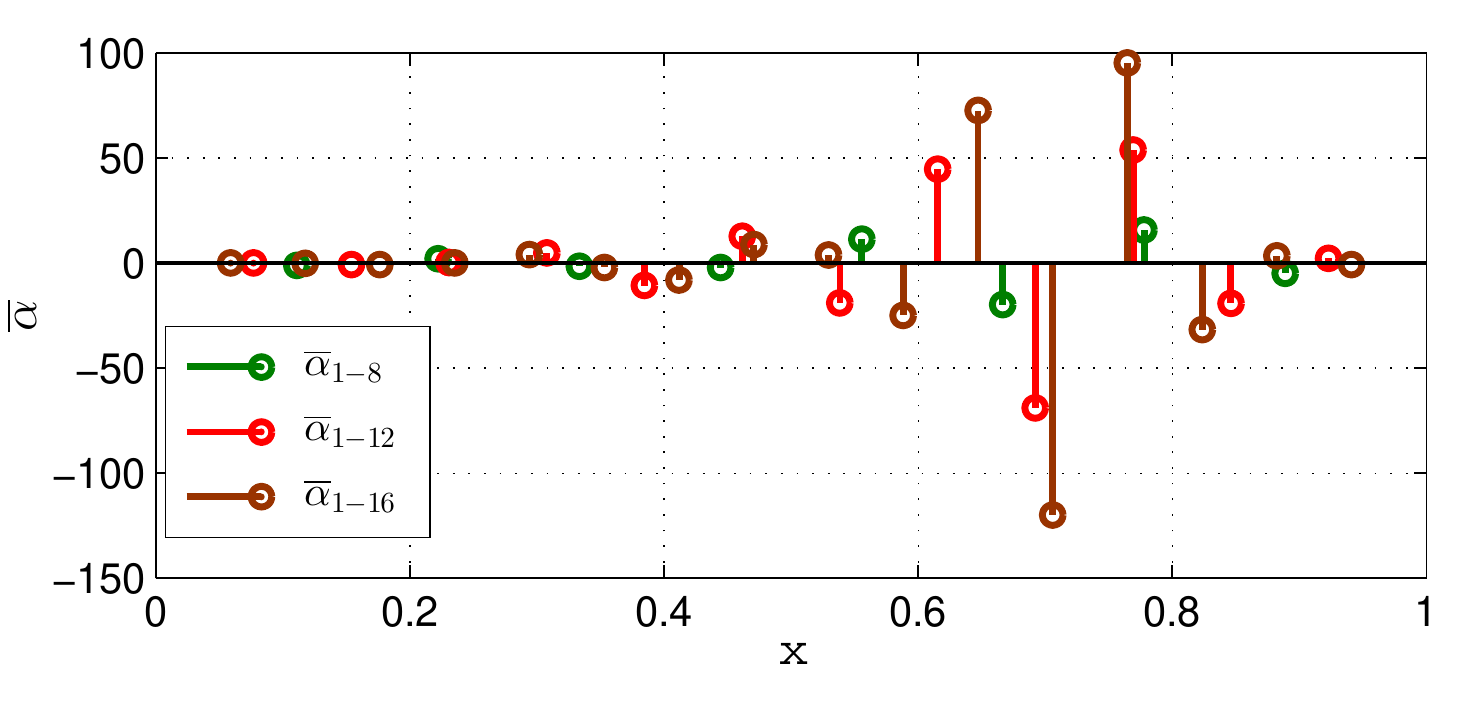}}
  \caption{Effect of number of actuators:
           (a) interpolation accuracy ($\widetilde{w}^d(x)$: desired shape; $\overline{w}^{d}_{n}(x)$: solution of the steady-state beam equation with $n=$8, 12 and 16 in-domain actuators);
           (b) amplitude of steady-state control signals for different setups.}
  \label{Fig: No of actuators}
\end{figure}

\begin{figure}[thpb]\
  \centering
  \includegraphics[scale=.45]{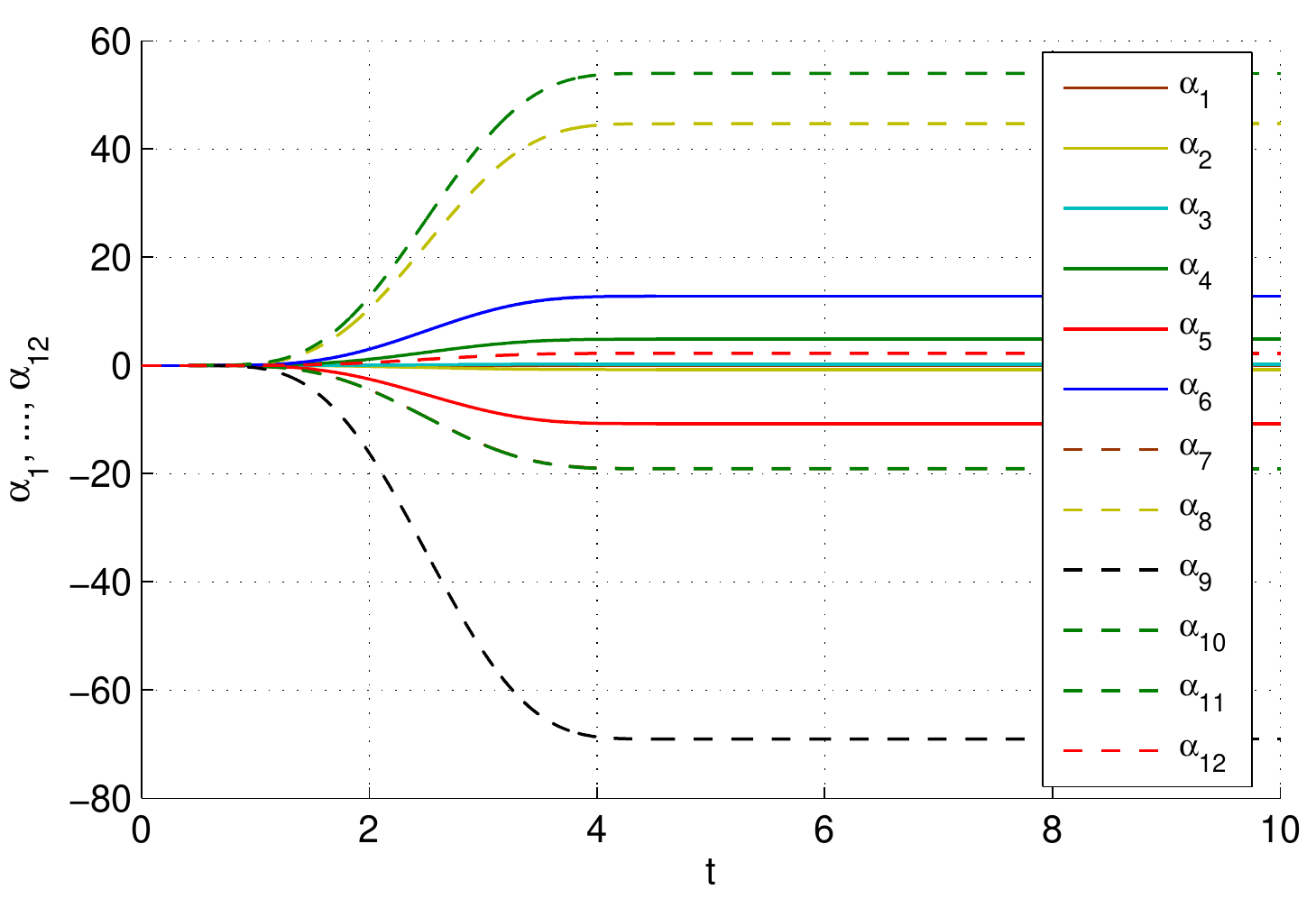}
  \caption{Feedforward control signals: $\overline{\alpha}_1, \ldots, \overline{\alpha}_{12}$.
  }
  \label{fig: controls}
\end{figure}

\begin{figure}[thpb]\
  \centering
  \subfigure[]{\label{Fig: 3D sol}\includegraphics[scale=.5]{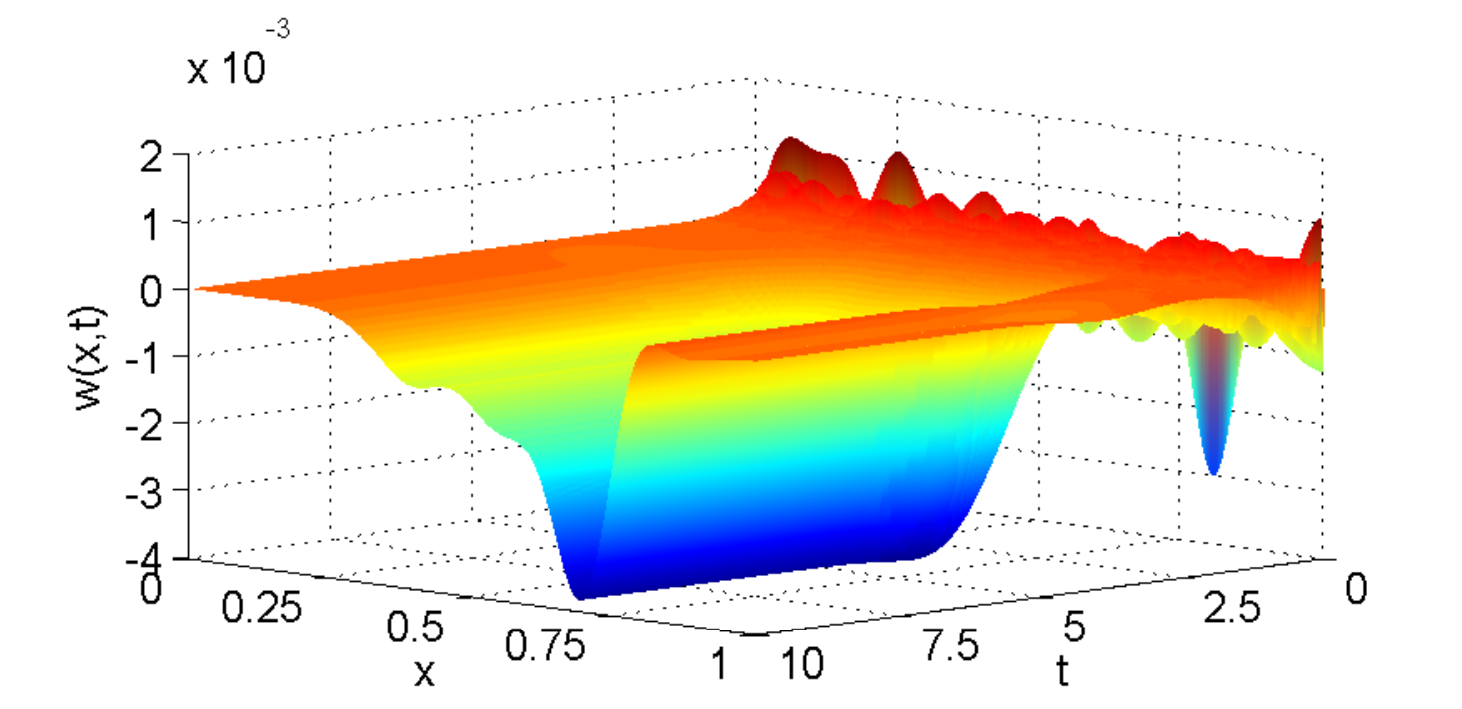}}
  \subfigure[]{\label{Fig: 3D error}\includegraphics[scale=.5]{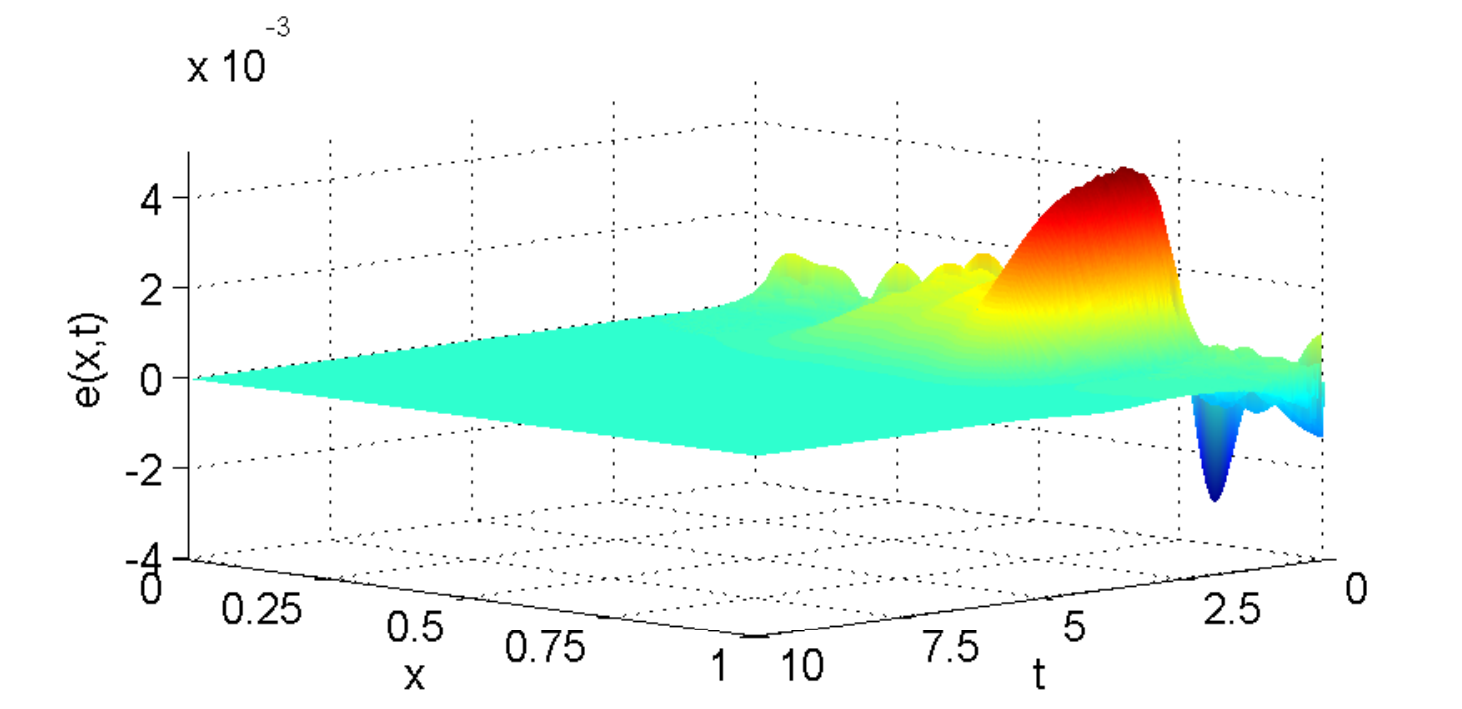}}
  \caption{Set-point control:
           (b) system response;
           (c) regulation error.}
  \label{Fig: beam deformation}
\end{figure}

\section{Conclusion}\label{Sec: Conclusion}
This paper presented a solution to the problem of in-domain control of a deformable beam described by an inhomogeneous PDE. A relationship between the original model and a system expressed in a standard boundary control form has been established. Flatness-based motion planning and feedforward control are then employed to explore the degrees-of-freedom offered by the system, while a closed-loop control is used to stabilize the system around the reference trajectories. The validity of the developed approach has been assessed through well-posedness and stability analysis. System performance is evaluated by numerical simulations, which confirm the applicability of the proposed approach. Note that in order to improve the resolution of manipulation, more actuators may be expected. Nevertheless, for the control of one-dimensional beams, the proposed scheme contains only one closed loop, allowing a drastic simplification of control system implementation and operation. This is an important feature for practical applications, such as the control of large-scale deformable micro-mirrors in adaptive optics systems.

\section{Acknowledgment}
This work is supported in part by a NSERC (Natural Science and Engineering Research Council of Canada) Discovery Grant. The second author is also supported in part by the Fundamental Research Funds for the Central Universities (\#682014CX002EM).

\section{Appendices}
\subsection{Proof of Proposition~\ref{proposition 2}}\label{Appendix: Proof of Proposition 2}
We prove the convergence of the power series \eqref{+17} using
Cauchy-Hadamard Theorem. The convergence of \eqref{+18} then follows
easily using the same argument.

Denote in \eqref{+16}:
\begin{equation}\label{b_{n}}
\begin{split}
 b_n=& \left(\sum_{k=0}^{n}\frac{x^{4k+1}}{(4k+1)!(4(n\!-\!k)+2)!} \right.\\
     &\left. -\sum_{k=0}^{n}\frac{x^{4k+3}}{(4k+3)!(4(n\!-\!k))!}\right)(-1)^{n}\phi^{(2n)}_{j}(t).
\end{split}
\end{equation}
Then, \eqref{+17} converges if
$\limsup_{n\rightarrow\infty}\sqrt[n]{|b_n|}<1$.

 Thus, for $0<x_j<1$ and $x \in \left[0,1\right]$, $b_{n}$ can be majorized as:
\begin{align*}
 |b_{n}| <&\left(\left| \sum_{k=0}^{n}\frac{1}{(4k+1)!(4(n\!-\!k)+2)!}\right| \right.\\
          &\left. +\left|\sum_{k=0}^{n}\frac{1}{(4k+3)!(4(n\!-\!k))!}\right|\right)
           \left|\phi_{j}^{(2n)}(t)\right|\notag\\
         <&\left(\left| \sum_{k=0}^{n}\frac{4k+2}{(4k+2)!(4(n\!-\!k))!}\right|\right.\\
          &\left.+\left|\sum_{k=0}^{n}\frac{4k+2}{(4k+2)!(4(n\!-\!k))!}\right|\right)
           \left|\phi_{j}^{(2n)}(t)\right|\notag\\
         \leq & 2(4n+2)\left|\sum_{k=0}^{n}\frac{1}{(4k+2)!(4(n\!-\!k))!}\right|\left|\phi_{j}^{(2n)}(t)\right|\notag\\
         <&2(4n+2)\frac{16^{n}}{(4n+2)!}\left|\phi_{j}^{(2n)}(t)\right|
         \leq \frac{2M}{K^{2n}}\frac{16^{n}((2n)!)^{\sigma}}{(4n+1)!}.
\end{align*}
Therefore,
\begin{align}
\limsup_{n\rightarrow\infty}\sqrt[n]{b_{n}}
\leq &
\limsup_{n\rightarrow\infty}16\left(\frac{2M}{K^{2n}}\right)^{1/n}\frac{((2n)!)^{\sigma/n}}{((4n+1)!)^{1/n}} \nonumber\\
=& \frac{16}{K^{2}}\limsup_{n\rightarrow\infty}\frac{\left(((2n)!)^{1/(2n)}\right)^{2n\sigma/n}}{\left(((4n+1)!)^{1/(4n+1)}\right)^{(4n+1)/n}}.
\end{align}
Applying Stirling's formula $\sqrt[n]{n!} \simeq (n/e)$ yields:
\begin{align}
\limsup_{n\rightarrow\infty}\sqrt[n]{b_{n}} \leq\frac{16e^{4-2
\sigma}}{K^{2}} \limsup_{n\rightarrow\infty}
\frac{(2n)^{2n\sigma/n}}{(4n+1)^{(4n+1)/n}}.\notag
\end{align}
Since
\begin{align}
 &\limsup_{n\rightarrow\infty}\frac{(2n)^{2n\sigma/n}}{(4n+1)^{(4n+1)/n}} \nonumber\\
=& \limsup_{n\rightarrow\infty}
\frac{(2n)^{2\sigma + \frac{3\sigma }{n} }}{(4n+1)^{4+\frac{1}{n}}}
=\left\{
   \begin{array}{ll}
     0, & \hbox{$\sigma<2$;} \\
     \dfrac{1}{16}, & \hbox{$\sigma=2$;} \\
      \infty, & \hbox{$\sigma>2$,}
   \end{array}
 \right.
\end{align}
we can conclude by Cauchy-Hadamard Theorem that
\eqref{+17} converges for $\sigma<2$ with a radius of convergence of infinity and for $\sigma=2$ if $K^{2} >1$.
The series \eqref{+17} diverges if $\sigma>2$.

\subsection{Proof of Claim~\ref{Cl: claim 1}}\label{Appendix: Proof of Claim 1}
Consider the first series in \eqref{expansion of I}. Fixing $m>0$, for $x,x_{j}\in (0,1)$, $j=1,2,\ldots,N$, we have
\begin{align*}
\bigg|\frac{(-1)^{k+1}m^{2k-1}}{(2k-1)}
\frac{(x-x_{j})^{2k+2}}{(2k+2)!} \bigg|\leq \frac{m^{2k-1}}{(2k-1)!}.
\end{align*}
Since $\lim_{n\rightarrow\infty}\frac{1}{\sqrt[n]{n!}}=0$, it follows that for any $a>0$
\begin{align*}
\lim_{n\rightarrow\infty}\sqrt[n]{\frac{a^n}{n!}}=
a\lim_{n\rightarrow\infty}\frac{1}{\sqrt[n]{n!}}=0.
\end{align*}
Thus
$
\sum_{k=0}^{\infty}\frac{m^{2k-1}}{(2k-1)!}
$ is convergent. We conclude that the first series in \eqref{expansion of I} is uniformly convergent. The convergence of the other terms in \eqref{expansion of I} can be proved in the same way.


\subsection{Proof of Claim~\ref{Cl: claim 2}}\label{Appendix: Proof of Claim 2}

We denoted by $\left[G(x_{i},\xi_{j})\right]_{N\times N}$ the matrix in the left-hand-side of \eqref{Eq: SS responses} formed by Green's functions. We argue by contradiction. If, otherwise, $\left[G(x_{i},\xi_{j})\right]_{N\times N}$ is not invertible, then it is of rank less than $N$. Without loss of generality, assume that there exist $N-1$ constants, $k_{1},\ k_{2}, ..., \ k_{N-1}$, such that
\begin{align*}
G(x_{1},\xi_{N})&= \sum_{i=1}^{N-1} k_{i}G(x_{1},\xi_{i}),\\
G(x_{2},\xi_{N})&= \sum_{i=1}^{N-1} k_{i}G(x_{2},\xi_{i}),\\
                &\;\;\vdots\\
G(x_{N},\xi_{N})&= \sum_{i=1}^{N-1} k_{i}G(x_{N},\xi_{i}).
\end{align*}
The above equations show that
\begin{equation}\label{equation G}
G(x,\xi_{N})
=  \sum_{i=1}^{N-1} k_{i}G(x,\xi_{i})
\end{equation}
has $N$ different positive solutions $x_{1},\ x_{2}, \cdots \ x_{N}$, $x_{i} \in (0,1)$, $i = 1, \ldots, N$.

We consider two cases:\\
(i) If $N>3$, since $\xi_{i}, i=1,\ldots, N,$ are distinguished, $G(x,\xi_{j})$, $i=1,\ldots, N,$ are all different from each other. Hence
\begin{equation*}
G(x,\xi_{N}) \not\equiv k_{1}G(x,\xi_{1})+k_{2}G(x,\xi_{2})+\cdots+k_{N-1}G(x,\xi_{N-1}).
\end{equation*}
Note that $G(x,\xi_{j}), j=1,\ldots, N,$ are of order at most $3$, then \eqref{equation G} has at most $3$ different solutions in $\mathbb{R}$, which is a contradiction. \\
(ii) If $N \leq 3$, it is easy to check that \eqref{equation G} has a solution
 $x=0$, and a pair of solutions $x=x^{0}$ and $x=-x^{0}$ near the origin $0$. By the assumption, \eqref{equation G} has $N$ different positive solutions, then it must be $N=1$, which leads to a contradiction with the non-invertible property of $G(x_{1},\xi_{1})\neq 0$.

\end{document}